\documentclass[12pt]{article}

\usepackage{amsthm,amsmath,amssymb,mathtools}
\usepackage{ascmac}
\usepackage{latexsym}
\usepackage{graphicx}
\usepackage{url}
\usepackage{hyperref}
\usepackage{xcolor}
\usepackage{bm}
\usepackage{comment}
\usepackage{statex}
\usepackage{diagbox}
\usepackage{multirow}
\usepackage{chngpage}
\usepackage{setspace} 

\usepackage{array}
\usepackage[width=0.9\textwidth]{caption}

\newtheorem*{theorem*}{Theorem}
\newtheorem{theorem}{Theorem}

\newtheorem*{prop*}{Proposition}
\newtheorem{rema}{Remark}
\newtheorem*{rema*}{Remark}
\newtheorem{lemma}[theorem]{Lemma}

\def\Real{\mbox{$\mathbb{R}$}}

\def\SymMat{\mbox{$\mathbb{S}$}}
\def\SymT{\mbox{$\mathbb{T}$}}

\newcolumntype{L}{>{\centering\arraybackslash}m{2.2cm}}

\def\0{\boldsymbol 0}
\def\1{\boldsymbol 1}
\def\2{\boldsymbol 2}
\def\3{\boldsymbol 3}
\def\4{\boldsymbol 4}
\def\5{\boldsymbol 5}
\def\6{\boldsymbol 6}
\def\7{\boldsymbol 7}
\def\8{\boldsymbol 8}
\def\9{\boldsymbol 9}

\def\b{\boldsymbol b}

\def\c{\boldsymbol c}

\def\e{\boldsymbol e}

\def\q{\boldsymbol q}

\def\v{\boldsymbol v}

\def\x{\boldsymbol x}
\def\y{\boldsymbol y}

\def\A{\boldsymbol A}

\def\H{\boldsymbol H}

\def\O{\boldsymbol O}
\def\P{\boldsymbol P}
\def\Q{\boldsymbol Q}

\def\W{\boldsymbol W}
\def\X{\boldsymbol X}
\def\Y{\boldsymbol Y}

\usepackage{tikz}
\usetikzlibrary{shapes,arrows}
\usetikzlibrary{calc}

\usepackage{mathtools}
\mathtoolsset{showonlyrefs=true}

\begin{document}

\title{ \Large Exploiting Aggregate Sparsity in Second Order Cone Relaxations for Quadratic Constrained Quadratic Programming Problems}

\author{
Heejune Sheen\thanks{School of Industrial and Systems Engineering, Georgia Institute of Technology, Atlanta, GA 30332, USA        
             ({\tt brianshn1@gmail.com}). This research was conducted while this author was a research student at the
		Department of  Mathematical and Computing Science,
            Tokyo Institute of Technology. 
              }   \and 
Makoto Yamashita\thanks{Department of  Mathematical and Computing Science,
            Tokyo Institute of Technology, 2-12-1 Oh-Okayama, Meguro-ku, Tokyo 152-8552, Japan        
             ({\tt Makoto.Yamashita@c.titech.ac.jp}).
	This research was partially supported by JSPS KAKENHI (Grant number: 18K11176).
 	}
}
\date{\normalsize November, 2019}


\maketitle 

\vspace*{-1cm}

\begin{abstract}
\noindent
Among many approaches to increase the computational efficiency of semidefinite programming (SDP) relaxation for quadratic constrained quadratic programming problems (QCQPs), exploiting the aggregate sparsity of the data matrices in the SDP by Fukuda et al.~(2001) and second-order cone programming (SOCP) relaxation have been popular. In this paper, we exploit the aggregate sparsity of SOCP relaxation of QCQPs. Specifically, we prove that exploiting the aggregate sparsity reduces the number of second-order cones in the SOCP relaxation, and that we can simplify the matrix completion procedure by Fukuda et al.~in both primal and dual of the SOCP relaxation problem without losing the max-determinant property. For numerical experiments, QCQPs from the lattice graph and pooling problem are tested as their SOCP relaxations provide the same optimal value as the SDP relaxations. We demonstrate that exploiting the aggregate sparsity improves the computational efficiency of the SOCP relaxation for the same objective value as the SDP relaxation, thus much larger problems can be handled by the proposed SOCP relaxation than the SDP relaxation.
\end{abstract}



\section{Introduction}
Quadratically constrained quadratic programming problems (QCQPs) represent an important class of optimization problems
in both theory and practice.
A variety of problems arising from engineering and combinatorial applications can be formulated as QCQPs, for example, 
quadratic assignment problem~\cite{Anstreicher2003}, radar detection~\cite{de2010fractional}, and graph theory~\cite{motzkin1965maxima}. More applications can be  found in~\cite{Bao2011}.
If QCQPs are convex, many efficient algorithms exist to find their solutions \cite{huang2006smoothing,nemirovskii1996extension}. Nonconvex QCQPs are, however, known as NP-hard in general~\cite{d2003relaxations}.

Nonconvex QCQPs have been studied by relaxation methods via lifting and convexification \cite{boyd2004convex}, 
most notably, semidefinite programming (SDP) relaxation. SDP relaxation has been popular as
it can obtain tight approximate optimal values of QCQPs. In fact, SDP relaxation has been applied to a broad range of problems such as the maxcut problems \cite{goemans1995improved}, sensor network localization \cite{biswas2006semidefinite}, optimal contribution selection
\cite{safarina2019conic, yamashita2018efficient}, and the pooling problem \cite{kimizuka2018solving, nishi2010semidefinite}.

Solving the SDP relaxation of large QCQPs can be very time-consuming and
obtaining  
an approximate optimal value with accuracy  is often difficult  due to numerical instability~\cite{safarina2019conic}. It is particularly true when the primal-dual interior-point methods \cite{mosek2015mosek, sturm1999using}
are used to solve the SDP relaxation. As a result, various methods have been proposed to alleviate the difficulties. 
The chordal sparsity exploitation proposed by 
Fukuda et al.~\cite{fukuda2001exploiting} for SDP problems is regarded as a  systematic method that effectively utilizes the
structure of the data matrices.
In~\cite{fukuda2001exploiting},
 the variable matrix  was decomposed into small sub-matrices, each of which was
 associated with the maximal cliques of the chordal graph of the SDP.  
To relate the resulting sub-matrices for the equivalence to the original SDP, 
 additional equality constraints were added to the SDP problem with the decomposed sub-matrices.
After solving the SDP with the sub-matrices and equality constraints, 
a completion procedure to patch the sub-matrices was performed 
to recover the original variable matrix as the final solution. Fukuda et al.~showed 
that the completion procedure results in a matrix with the maximum determinant among all possible completed matrices.
From the computational perspective,
the computational gain by exploiting the chordal sparsity in SDPs is clear  only 
when the resulting SDPs have small sizes of the sub-matrices and  moderate numbers of  
additional equality constraints. 

The chordal sparsity in SDP relaxation has been studied and implemented in many literature and softwares.
The chordal sparsity was studied from various angles in ~\cite{vandenberghe2015chordal} by
  Vandenberghe and Andersen. 
Kim et al.~\cite{kim2011exploiting} introduced a Matlab software  package called
SparseCoLO~\cite{fujisawa2009user} that exploits the chordal sparsity of matrix inequalities.
SDPA-C~\cite{yamashita2015fast} is  implementation of a parallel approach for SDPs using the chordal sparsity.
Mason et al.~\cite{mason2014chordal} applied the chordal sparsity to the Lyapunov equation.
In contrast,
 the chordal sparsity has not been studied in the context of second-order cone programming (SOCP) relaxation.

For some classes of QCQPs, the SOCP relaxation provides the same  optimal value as the SDP relaxation, though it is a weaker relaxation than the SDP relaxation in general. Kim and Kojima~\cite{kim2003exact} proved that nonconvex QCQPs with
non-positive off-diagonal elements can be exactly solved by the SDP or SOCP relaxations.
More recently, for QCQPs with zero diagonal elements in the data matrices, 
Kimizuka et al.~\cite{kimizuka2018solving} showed that the SDP, SOCP and linear programming (LP) relaxations 
provide the same objective value.
For these classes of QCQPs, 
the SOCP relaxations are far more efficient than the SDP relaxation. 

The main purpose of this paper is 
to exploit the aggregate  sparsity in the SOCP relaxation. 
Our approach in this paper has two major differences from 
the method utilizing the chordal sparsity in SDP~\cite{fukuda2001exploiting}. 
First, it does not require any additional equality constraints to exploit the sparsity in the SOCP. 
Thus, the proposed  method for the SOCP can improve the computational efficiency for
 solving the SOCP relaxation. 
The problems, for which  the SDP and SOCP relaxations provide the same optimal value,
 especially can benefit  from  the increased efficiency of the proposed method.
Second, our approach can generate a completed matrix that attains the maximum determinant without relying on the completion procedure as in the SDP relaxation. As a result,  
the completion procedure is not necessary in our approach. 
Thus, we can expect that
the SOCP relaxation can be solved much faster.

We report numerical results to demonstrate the efficiency improved by exploiting the sparsity in the SOCP relaxation. 
We generated test instances to satisfy the condition of \cite{kim2003exact} using lattice graphs, and we also tested instances of the pooling problems from \cite{kimizuka2018solving}. Through the numerical experiments, we observe that the SOCP relaxation with the proposed sparsity exploitation spends much less computational time to obtain the same objective value as the SDP relaxations. 
	
The rest of the paper is organized as follows. In Section~\ref{sec:pre}, we briefly review  relaxation methods for QCQPs and
describe some background
on exploiting the chordal sparsity in SDPs.  In Section~\ref{sec:sparsity}, we discuss the sparsity exploitation in the SOCP relaxation. We prove 
 that the completed matrix in the proposed approach attains the maximum determinant among possible completed matrices.  In Section~\ref{sec:experiments}, we report the numerical results on the QCQPs from the lattice  graph and  pooling problem.
Finally, in Section~\ref{sec:conclusion}, we give our conclusion remarks.
\section{Preliminaries}\label{sec:pre}
	 We start this section by introducing some notation that will be  used in this paper. We use the superscript $\top$ to denote the transpose of a matrix or a vector.
	Let $\mathbb{R}^n$ be a $n$-dimensional Euclidean space. We use $\mathbb{R}_{+}^n \subset \mathbb{R}^n$ for a nonnegative orthant in $n$-dimensional Euclidean space. 
Let $\mathbb{R}^{n \times n}$ be the set of $n \times n$ real matrices, and let $\mathbb{S}^n \subset \Real^{n \times n}$ be the set of $n \times n$ real symmetric matrices. We use $\mathbb{S}_{+}^n (\mathbb{S}_{++}^n) \subset \mathbb{S}^n$ for the set of positive semidefinite matrices (positive definite matrices, respectively) of dimension $n$. 
We also use $\X \succeq \O (\X \succ \O)$ to denote 
$\X$ is a positive semidefinite matrix (a positive definite matrix, respectively).
The inner product $\X \bullet \Y$ between $\X$ and $\Y$ in $\mathbb{S}^n$ is defined by $\X \bullet \Y = trace(\X \Y)$. 
	
\subsection{SDP and SOCP relaxations of QCQPs}

A general form of QCQPs can be given as follows:	
\begin{align}
\begin{array}{rcl}
\mbox{minimize} &: &\x^\top \P_{0} \x +2\q_{0}^\top \x+r_{0} \\
\mbox{subject to} &: &\x^\top \P_{k} \x +2\q_{k}^\top \x+r_{k} \leq 0  \ \text{for} \  1 \le k \le m,
\end{array} \label{QCQP1}
\end{align}
where the decision variable is $\x \in \Real^n$, and the input data are $\P_{k} \in \SymMat^n$, $\q_{k} \in \mathbb{R}^n$ and $r_{k} \in \mathbb{R}$ for $ k = 0, \ldots, m$.
	
If $\P_{k} \succeq \O$ for all $k$, the problem~\eqref{QCQP1} is  a convex problem and can be solved efficiently, for example, by interior-point methods~\cite{nesterov1994interior}. 
In contrast, if $\P_{k}$  is not positive semidefinite matrix for some $k$, the problem is  nonconvex and can be NP~\cite{d2003relaxations}.
		
We define $\Q_{k} := \left[ \begin{matrix} r_{k} & \q_{k}^\top \\ \q_{k} & \P_{k} \end{matrix} \right] \in \SymMat^{n+1}$ for each $k = 0, \ldots, m$, and introduce a variable matrix $\X \in \SymMat^{n+1}$.  The
standard SDP relaxation based on lift-and-project convex relaxation \cite{balas2005projection, grotschel2012geometric, lovasz1991cones} for QCQP~\eqref{QCQP1}
can be given as follows:
\begin{align} 
\mbox{minimize:} \ \ &\Q_{0} \bullet \X \nonumber \\
\mbox{subject to:} \ \ &\Q_{k} \bullet \X \leq 0 \ \text{for} \ 1 \le k \le m  \nonumber \\
& \H_0 \bullet \X  = 1 \nonumber \\
&\X  \in \SymMat_{+}^{n+1},  \label{SDPrelax}
\end{align}
where $\H_{0} := \left[ \begin{array}{cc} 1 & \0^\top \\ \0 & \O \end{array} \right] \in \SymMat^{n+1}$.

%
%

By further relaxing the positive semidefinite condition $\X \in \SymMat_{+}^{n+1}$ with 
the positive semidefinite conditions
of all $2 \times 2$ principal sub-matrices, Kim and Kojima \cite{kim2001second} proposed an SOCP relaxation \eqref{SOCPrelax} below. This corresponds to the dual of
the first level of the scaled diagonally dominant sum-of-squares (SDSOS) relaxation hierarchy discussed in \cite{ahmadi2014dsos}.
	\begin{align}
\mbox{minimize:} \ \ &\Q_{0} \bullet \X \nonumber \\
\mbox{subject to:} \ \ &\Q_{k} \bullet \X \leq 0 \ \text{for} \ 1 \le k \le m  \nonumber \\
& \H_{0} \bullet \X = 1 \nonumber \\
&\X  \in \SymT_{+}^{n+1}.	\label{SOCPrelax}
\end{align}
Here, $\SymT_+^{n+1}$ is defined to denote the set of symmetric matrices that satisfy
the positive semidefinite conditions of all $2 \times 2$  principal sub-matrices as 
\begin{align}
\SymT_{+}^{n+1} := \{\X \in \SymMat^{n+1} | \X^{ij} \succeq \O
\ \text{for} \ (i,j) \in J\}.
\label{eq:TC}
\end{align}
where $\X^{ij}$ is defined by
\begin{align}
\X^{ij} := \left(\begin{array}{cc}
X_{ii} & X_{ij} \\ X_{ij} & X_{jj} \end{array} \right)
\in \SymMat^2
\end{align}
and $J$ is an index set defined by $J := \{(i,j) | 1 \le i < j \le n+1\}$.
The equivalence 
between 
$\X^{ij}  \succeq  \O$ and
$\frac{X_{ii} +X_{jj}}{2} \geq \left\Vert \left( \begin{array}{c} \frac{X_{ii} -X_{jj}}{2} \\ X_{ij} \end{array} \right)  \right\Vert$
enables us to formulate \eqref{SOCPrelax} problem as an SOCP~\cite{kim2001second},
therefore, \eqref{SOCPrelax} is an SOCP relaxation of \eqref{QCQP1}.

Due to the relation $\SymMat_+^{n+1} \subset \SymT_+^{n+1}$,
the SOCP relaxation \eqref{SOCPrelax} is weaker than the SDP relaxation \eqref{SDPrelax}.
However, there exist certain classes of QCQPs whose SDP and SOCP relaxations are  exact, 
for example, \cite{kim2003exact}. 
	

\subsection{The matrix completion with chordal graph}
It is frequently observed  in many applications that
the matrices $\Q_0, \Q_1, \ldots, \Q_m$  have some structural sparsity.
Fukuda et al.~\cite{fukuda2001exploiting} and Nakata et al.~\cite{nakata2003exploiting} exploited the sparsity in the data matrices
using the chordal graph. We call such  sparsity related to the chordal graph the chordal sparsity.
We briefly describe their matrix completion technique.

For the SDP relaxation problem \eqref{SDPrelax}, we first define 
an aggregate sparsity graph $G(V,\overline{E})$ as an undirected graph
with the set of vertices $V = \{1, 2, \ldots, n+1\}$ 
and  the set of edges 
\begin{equation}
\overline{E} = \left\{(i,j) \in J \ \middle| \   \text{and} \ 
\ [\Q_k]_{ij} \neq 0 \ \text{for some} \ k \in \{0, 1, \ldots, m\}\right\}.
\end{equation}
We sometimes call $\overline{E}$ the aggregate sparsity pattern,
and the sparsity related to $\overline{E}$ the aggregate sparsity.
An undirected graph is called \textit{chordal} if every cycle of four or more vertices has a chord.
When the aggregate sparsity graph $G(V,\overline{E})$ is not chordal,   
we can find a chordal graph $G(V, \widehat{E})$ such that $\overline{E} \subset \widehat{E} \subset J$
by adding appropriate edges to $\overline{E}$.
The graph $G(V, \widehat{E})$ is called a chordal extension, and
it is known that the chordal extension is related with the sparse Cholesky factorization~\cite{vandenberghe2015chordal}.

A vertex set $C \subset V$ is called a clique if the induced subgraph 
$G(C, (C \times C) \cap \overline{E})$ is a complete graph,
and a clique $C$ is called a maximal clique if it is not a subset of any other clique.
When $G(V, \widehat{E})$ is a chordal graph, we can enumerate the set of maximal cliques $\Lambda = \{C_1, \ldots, C_p\}$, where $p$ is the number of maximal cliques. Without loss of generality, we can assume that the order of 
$\{C_1, \ldots, C_p\}$ is a perfect elimination ordering~\cite{fukuda2001exploiting, nakata2003exploiting},
and that each vertex in $V$ is covered by at least one maximal clique, that is, $V =  \cup_{l=1}^p C_l$.

We use the following notation for an edge set $E \subset J$ and a matrix $\overline{\X} \in \SymMat^{n+1}$: 
\begin{align}
\begin{array}{lcl}
\SymMat^{n+1}(E, \overline{\X}, ?) &=& \{ \X \in \SymMat^{n+1} | X_{ij} = \overline{X}_{ij} \ \text{for} \ (i,j) \in E \cup D\} \\
 \SymMat_{+}^{n+1} (E, ?) &=& \{\X \in \SymMat^{n+1} | 
  \exists \widehat{\X} \in \SymMat_{+}^{n+1} \  \text{such that} \ \widehat{\X} \in \SymMat^{n+1}(E, \X, ?) \}  \\
\SymMat^{n+1}(E,0) &=& \{\X \in \SymMat^{n+1} | X_{ij} = X_{ji} = 0 \  \text{for} \  (i,j) \notin E \cup D \} \\
\SymMat_{+}^{n+1}(E,0) &=& \SymMat^{n+1}(E,0) \cap \SymMat_{+}^{n+1}. 
\end{array}
\end{align}
Here, $D$ is the index set that corresponds to the diagonal elements;
$D = \{(i,i) : 1 \le i \le n+1 \}$.
We also define $\overline{\X}(E, C)$  as a matrix in  $\SymMat^{n+1}(E, \overline{\X}, ?) \cap \SymMat^{n+1}((C \times C) \cap J, 0)$.
Note that if $E$ is the edge set of the chordal extension (that is, $E=\widehat{E}$) and $C_l$ is a clique of $G(V, \widehat{E})$, then $\overline{\X} (\widehat{E}, C_l)$ is uniquely determined for a given $\overline{\X} \in \SymMat^{n+1}$.

The fundamental theorem on the matrix completion by Grone et al.~\cite{grone1984positive}
can be described as follows.
\begin{theorem} {\upshape \cite{grone1984positive}}
    Fix $\overline{\X} \in \SymMat^{n+1}$.
 For $\X \in \SymMat^{n+1}(\widehat{E},\overline{\X},?)$, it holds that $\X \in \SymMat_+^{n+1}(\widehat{E},?)$ if and only if $\overline{\X}(\widehat{E}, C_{l}) \succeq \O$ for each $C_{l}  \in \Lambda$.
\end{theorem}
The main idea in \cite{fukuda2001exploiting} and \cite{nakata2003exploiting}
is to replace the positive semidefinite condition $\X \succeq \O$ with
the positive semidefinite conditions for the sub-matrices $\X(\widehat{E}, {C_1}), \ldots, \X(\widehat{E}, C_p) \succeq  \O$. 
More precisely,
the following SDP \eqref{SDPrelaxDC} is solved instead of \eqref{SDPrelax}:
	\begin{align} 
\mbox{minimize:} \ \ &\sum_{l=1}^p \Q_{0}(\widehat{E}, C_{l}) \bullet \X(\widehat{E}, C_{l}) \nonumber \\
\mbox{subject to:} \ \ &\sum_{l=1}^p \Q_{k}(\widehat{E}, C_{l}) \bullet \X(\widehat{E}, C_{l}) \leq 0  \ (k=1,..., m)  \nonumber \\
& [\X(\widehat{E}, C_{u})]_{11} = 1 \ \text{for} \ u \in \left\{u \in \{1, 2, \ldots, p\} | 1 \in C_u\right\}\nonumber \\
&[\X(\widehat{E}, C_{u})]_{ij} = [\X(\widehat{E}, C_{v})]_{ij}  \ \text{for} \ (i,j,u,v) \in U \nonumber \\
&\X(\widehat{E}, C_{l})  \succeq \O \ (l=1,..., p).  \label{SDPrelaxDC}
\end{align}
For each $k$, $\Q_k$ is decomposed into 
$\Q_{k}(\widehat{E}, C_1), \ldots, \Q_{k}(\widehat{E}, C_p) \in \SymMat^{n+1}$ such that 
$\Q_{k} = \sum_{l=1}^p \Q_{k}(\widehat{E}, C_{l})$.
It is always possible to decompose as above
since $\widehat{E}$ is a chordal extension of the aggregate sparsity pattern graph.
The set $U$ defined by
\begin{align}
\begin{array}{rcl}
	U = \{(i,j,u,v) &|& (i,j) \in (C_{u} \times C_{u}) \cap  (C_{v} \times C_{v}) \}
    \backslash \{(1,1)\}, \ i<j, \\ 
	& & (C_{u} \times C_{u}) \cap  (C_{v} \times C_{v}) \neq \emptyset \  \ \mbox{for} \ \ 1 \leq u < v \leq p  \}.
    \end{array}
    \end{align}
 represents the overlaps among the cliques, thus
the additional equality constraints 
$[\X(\widehat{E}, C_{u})]_{11} = 1$ and 
$[\X(\widehat{E}, C_{u})]_{ij} = [\X(\widehat{E}, C_{v})]_{ij}$ should be
 introduced in \eqref{SDPrelaxDC} for the overlapped elements.

We use $\overline{\X}(\widehat{E}, C_1), \ldots, \overline{\X}(\widehat{E}, C_l)$ to denote a solution of \eqref{SDPrelaxDC}. To simplify the discussion here, we assume
$\overline{\X}(\widehat{E}, C_1), \ldots, \overline{\X}(\widehat{E}, C_l)$ are positive definite matrices.
With the additional equality constraints above, we can uniquely determine the entire matrix $\overline{\X} \in \SymMat^{n+1}$ such that 
\begin{align}
\overline{X}_{ij} = 
\left\{\begin{array}{lcl}
 {[\overline{\X}(\widehat{E}, C_l)]}_{ij} & \text{if} & (i,j) \in C_l \times C_l \ \text{for some } l \in \{1, \ldots, p\} \\
0 & \text{if} & (i,j) \notin 
(C_1 \times C_1) \cup \ldots \cup (C_p \times C_p).
\end{array}\right.
\end{align}

Since the entire $\overline{\X}$ is not necessarily positive semidefinite, 
we complete $\overline{\X}$ to $\widehat{\X} \in \SymMat_{+}(\widehat{E}, ?)$
by the completion procedure of \cite{fukuda2001exploiting}. 
The completion procedure is  described in detail in 
\cite{nakata2003exploiting} 
where the matrix is  completed in the order of the maximal cliques 
$\{C_1, \ldots, C_p\}$.
We include the following lemma on $\widehat{\X}$ for the subsequent discussion.

\begin{lemma}\label{lem:completion}
    {\upshape \cite{fukuda2001exploiting}}.
    The completed matrix $\widehat{\X}$ computed by the completion procedure
    is the unique positive definite matrix that maximizes the determinant 
    among all possible matrix completions of $\overline{\X}$,
    that is,
    \begin{align}
    \det \widehat{\X} = \max \left\{\det \X \ \middle| \ \X \in \SymMat_{++}^{n+1} \cap \SymMat^{n+1}(\widehat{E}, \overline{\X}, ?) \right\}.
    \end{align}    
\end{lemma}

It follows from \cite{fukuda2001exploiting} that the completed matrix $\widehat{\X}$ is an optimal solution of the SDP relaxation problem~\eqref{SDPrelax}.

We should mention that 
there exists trade-off in terms of the computational efficiency for solving \eqref{SDPrelaxDC}. 
Replacing $\X \succeq \O$ 
with $\X(\widehat{E}, C_1), \ldots, $ $ \X(\widehat{E}, C_p) \succeq  \O$ reduces the computational cost required 
for 
 $\X \succeq \O$, especially when  the size of $\X$ is large.  
 However, the additional equality constraints such as $[\X(\widehat{E}, C_{u})]_{11} = 1$ and 
$[\X(\widehat{E}, C_{u})]_{ij} = [\X(\widehat{E}, C_{v})]_{ij}$ can be new computational burden. Moreover,
the completion procedure needs to be performed. 
Thus,  the conversion to \eqref{SDPrelaxDC} works well
when the cliques $C_1, \ldots,  C_p$ are small and 
the size of the set  $U$ for the overlapping elements is small.

\section{The sparsity of the SOCP relaxation}
\label{sec:sparsity}

In this section, we discuss how  the aggregate sparsity 
in the SOCP relaxation \eqref{SOCPrelax} can be exploited and how it is different 
from the case in the SDP relaxation.
We will also investigate the sparsity exploitation using dual problems.

\subsection{A matrix completion in the SOCP relaxation}\label{sec:mc-socp}
Similarly to the chordal sparsity in the SDP relaxation, 
we define the following notation for an edge set $E \subset J$: 
\begin{equation}
\begin{array}{lcl}
\SymT_{+}^{n+1} (E, ?) &:=& \{\X \in \SymMat^{n+1} | 
\exists \widehat{\X} \in \SymT_{+}^{n+1} \  \text{such that} \ \widehat{\X} \in \SymMat^{n+1}(E, \X, ?) \}  \\
\SymT_{+}^{n+1}(E,0) &:=& \SymMat^{n+1}(E,0) \cap \SymT_{+}^{n+1}. 
\end{array}
\end{equation}
The set $\overline{\SymT}_+ (E)$ for an edge set $E$ is defined by
\begin{equation}
\overline{\SymT}_+^{n+1} (E) := \{\X \in \SymMat^{n+1} |
\X^{ij} \succeq \O \ \text{for} \ (i,j) \in E,  X_{ii} \ge 0 \ \text{for} \ i \in V_e \}, 
\end{equation}
where 
\begin{equation}
\begin{array}{rcl}
V_e &:=& \{i \in V | 
(i,j) \notin E \ \text{for} \ \forall j \in \{i+1, \ldots, n+1\} \\

& & \text{and} \  (j,i) \notin E \ \text{for} \ \forall j \in \{1, \ldots, i-1\}
\} 
\end{array}
\end{equation}
is the set of vertices that are not involved in $E$.
Since $\X^{ij} \succeq \O$ guarantees the non-negativeness of  $X_{ii}$
and $X_{jj}$ for $(i,j) \in E$, we know that all the diagonal elements of any matrix in 
$\overline{\SymT}_+^{n+1} (E)$ are nonnegative, that is, 
$X_{ii} \ge 0$ holds for each $i \in V$ 
if $\X \in \overline{\SymT}_+^{n+1} (E)$.

We now examine the equivalence between $\SymT_{+}^{n+1}(E, ?)$ and $\overline{\SymT}_+^{n+1} (E)$.
For a matrix $\X \in \overline{\SymT}_+^{n+1}(E)$, let us consider a range $R_{ij}(\X) := [-\sqrt{X_{ii} X_{jj}}, \sqrt{X_{ii} X_{jj}}]$ for each $(i,j) \in J$ and
 introduce a set 
\begin{equation}
\widehat{\SymT}_+^{n+1} (E, \X) := \{\widehat{\X} \in \SymMat^{n+1}(E, \X, ?)  | 
\widehat{X}_{ij} \in R_{ij}(\X) \ \text{for} \ (i,j) \notin E \cup D\}.
\end{equation}

\begin{lemma}\label{lem:SymT}
    It holds that 
    $\SymT_{+}^{n+1}(E, ?) = \overline{\SymT}_+^{n+1} (E)$.
\end{lemma}

\begin{proof}
    First, we show $\SymT_{+}^{n+1}(E, ?) \subset \overline{\SymT}_+^{n+1} (E)$. 
   Fix $\X \in \SymT_{+}^{n+1}(E, ?)$. Then, there exists $\widehat{\X} \in \SymT_+^{n+1}$ such that $\widehat{X}_{ij} = X_{ij}$ for $(i,j) \in E \cup D$.
   Since  $\widehat{\X} \in \SymT_+^{n+1}$ and $E \subset J$, 
   we know $\left[\begin{array}{cc}\widehat{X}_{ii} & \widehat{X}_{ij} \\
   \widehat{X}_{ij} & \widehat{X}_{jj} \end{array} \right] \succeq \O$ for $(i,j) \in E$,
thus $\widehat{X}_{ij} = X_{ij}$ for $(i,j) \in E \cup D$ leads to $\X^{ij} \succeq \O$.
Furthermore, for each $i \in V \backslash \{n+1\}$, it holds that $(i, i+1) \in J$.
Then, $\left[\begin{array}{cc}\widehat{X}_{ii} & \widehat{X}_{i,i+1} \\
\widehat{X}_{i,i+1} & \widehat{X}_{i+1, i+1} \end{array} \right] \succeq \O$ guarantees 
$\widehat{X}_{ii} \ge 0$ and $\widehat{X}_{i+1, i+1} \ge 0$ for each $i \in V\backslash \{n+1\}$. This implies $X_{ii} \ge 0$ for each $i \in V_e$,
since  $\widehat{X}_{ii} = X_{ii}$ for $(i,i) \in D$ and $V_e \subset V$.
    
For 
$\SymT_{+}^{n+1}(E, ?) \supset \overline{\SymT}_+^{n+1} (E)$, we
fix $\X \in \overline{\SymT}_+^{n+1} (E)$ and take any matrix 
$\widehat{\X}$ from $\widehat{\SymT}_+^{n+1} (E, \X)$.
Then, we can show $\widehat{\X} \in \SymT_{+}^{n+1}$.
In particular, if $(i,j) \notin E \cup D$, then we have
$\left[\begin{array}{cc}\widehat{X}_{ii} & \widehat{X}_{ij} \\
\widehat{X}_{ij} & \widehat{X}_{jj} \end{array} \right] 
= \left[\begin{array}{cc}X_{ii} & \widehat{X}_{ij} \\
\widehat{X}_{ij} & X_{jj} \end{array} \right]\succeq \O$,
since $-\sqrt{X_{ii} X_{jj}} \le \widehat{X}_{ij} \le \sqrt{X_{ii} X_{jj}}$.

\end{proof}

\begin{rema}\label{rema:SOCP-sparsity}
    \upshape
   From the proof of Lemma~\ref{lem:SymT}, we observe that 
    for  a given matrix $\overline{\X} \in \overline{\SymT}_+^{n+1}(\overline{E})$ 
    corresponding to the aggregate sparsity graph $G(V,\overline{E})$,
      $\overline{\X}$  can be completed
      to  some matrix $\widehat{\X} \in \widehat{\SymT}_+^{n+1}(\overline{E}, \overline{\X})$ without changing 
    the elements specified in $\overline{E}$.    
    In addition,  $\widehat{\SymT}_+^{n+1}(\overline{E}, \overline{\X})$ covers
    all possible completion matrices of $\overline{\X}$ in $\SymT_+^{n+1}$.
\end{rema}
As a result of the observation in Remark \ref{rema:SOCP-sparsity},
we can modify \eqref{SOCPrelax} as 
the following SOCP~\eqref{SOCPsparsity}. 
Notice that $\SymT_+^{n+1}$ is replaced with
$\overline{\SymT}_{+}^{n+1}(\overline{E})$.
\begin{align}
\mbox{minimize:} \ \ &\Q_{0} \bullet \X \nonumber \\
\mbox{subject to:} \ \ &\Q_{k} \bullet \X \leq 0 \ \text{for} \ 1 \le k \le m  \nonumber \\
& \H_{0} \bullet \X = 1 \nonumber \\
&\X  \in \overline{\SymT}_{+}^{n+1}(\overline{E}).	\label{SOCPsparsity}
\end{align}

Let $\overline{\X}$ be an optimal solution of \eqref{SOCPsparsity}.
We also let $\zeta$ and $\overline{\zeta}$ 
be the optimal values of 
\eqref{SOCPrelax} and \eqref{SOCPsparsity}, respectively.
Since $\SymT_+^{n+1} \subset 
\overline{\SymT}_+^{n+1}(\overline{E})$ 
from the relation $\overline{E} \subset J$,  we have
$\zeta \ge \overline{\zeta}$, that is, 
\eqref{SOCPsparsity} is a further relaxation of \eqref{SOCPrelax} in general.
However, we can show the equivalence between $\zeta$ and $ \overline{\zeta}$.  

\begin{theorem}\label{theo:SOCPsparsity}
	 For  $\zeta$ and $ \overline{\zeta}$, we have
	  $\zeta =  \overline{\zeta}$. 
	In addition, suppose that $\X^*$ and $\overline{\X}$ are optimal solutions 
	of \eqref{SOCPrelax} and \eqref{SOCPsparsity}, respectively.
	Then, $\X^*$ is an optimal solution of \eqref{SOCPsparsity},
	and $\widehat{\SymT}_+^{n+1}(\overline{E}, \overline{\X})$ is included in the set of optimal solutions of \eqref{SOCPrelax}.
\end{theorem}

\begin{proof}
    First, we show that any matrix $\widehat{\X} \in \widehat{\SymT}_+^{n+1}(\overline{E}, \overline{\X})$ is an optimal solution 
    of \eqref{SOCPrelax}.
    From Remark~\ref{rema:SOCP-sparsity}, it follows that 
    $\Q_0 \bullet \widehat{\X} = \Q_0 \bullet \overline{\X} = \overline{\zeta}$,
    $\Q_k \bullet \widehat{\X} = \Q_k \bullet \overline{\X}$ for $k = 1, \ldots, m$,
    and $\H_0 \bullet \widehat{\X} = \H_0 \bullet \overline{\X}$. 
    As discussed in the proof of Lemma~\ref{lem:SymT},  we have
   $ \widehat{\X} \in \SymT_+^{n+1}$,
   therefore, $\widehat{\X}$ is a feasible solution of \eqref{SOCPrelax}.
   Since its objective value $\Q_0 \bullet \widehat{\X}$ is $\overline{\zeta}$,
   $\overline{\zeta} \ge \zeta$ holds.
   By combining this with $\zeta \ge \overline{\zeta}$,
   we know that $\zeta = \overline{\zeta}$, and
    this implies $\widehat{\X}$ is an optimal solution of \eqref{SOCPrelax}.

	Finally, $\X^*$ is a feasible solution of \eqref{SOCPsparsity}
	from $\SymT_+^{n+1} \subset \overline{\SymT}_+^{n+1}(\overline{E})$,
	and its objective value is $\zeta$. Therefore,  $\zeta = \overline{\zeta}$ leads to the conclusion that $\X^*$ is also an optimal solution of \eqref{SOCPsparsity}.    
\end{proof}

Lemma~\ref{lem:completion} from \cite{fukuda2001exploiting} shows that 
the completed matrix by the completion procedure has the maximum determinant among all possible matricies.
To discuss a similar maximum-determinant property in the framework of \eqref{SOCPrelax}, we 
need to introduce the determinant 
for $\SymT_+^{n+1}$.

From the   self-concordant function  discussed in \cite{nesterov1994interior} 
  for the theoretical analysis of interior-point methods,
we see that the standard self-concordant barrier function at $\X \in \SymMat_{++}^{n+1}$ 
takes the form of $-\log\det \X$.
In $\SymT_+^{n+1}$, we have multiple positive semidefinite 
matrices $\X^{ij}$ for $(i,j) \in J$, thus, if  interior-point methods are applied
to \eqref{SOCPrelax}, 
$\sum_{(i,j) \in J} (-\log\det \X^{ij}) 
= -\log \left(\Pi_{(i,j) \in J} \det \X^{ij}\right)$
can be used as a self-concordant barrier function.
Using the analogy,
we define the determinant for $\X \in \SymT_+^{n+1}$
by 
\begin{equation}
{\det}_{\SymT} \X := \Pi_{(i,j) \in J} \det \X^{ij}. \label{eq:detT}
\end{equation}
Note that $\X^{ij} \succeq  \O$ is equivalent to a second-order cone constraint 
$\frac{X_{ii} +X_{jj}}{2} \geq \left\Vert \left( \begin{array}{c} \frac{X_{ii} -X_{jj}}{2} \\ X_{ij} \end{array} \right)  \right\Vert$. 
Let 
 ${\det}_{\text{SOCP}}$ denote the determinant of $\v = (v_0, v_1, v_2)^\top \in \mathbb{R}$ for a second-order cone
$v_0 \ge \sqrt{v_1^2 + v_2^2}$ defined in \cite{alizadeh2003second}. Then,  
$\det_{\text{SOCP}} (\v) = v_0^2 - v_1^2 - v_2^2$,
and 
$\det_{\text{SOCP}}\left( \left(\frac{X_{ii} +X_{jj}}{2},  \frac{X_{ii} -X_{jj}}{2},  X_{ij}\right)^\top \right) =  \det \X^{ij}$ 
follows.
Thus, the definition of ${\det}_{\SymT}$ in \eqref{eq:detT} is also consistent with 
$\det_{\text{SOCP}}$ in \cite{alizadeh2003second}.

For $\overline{\X} \in \overline{\SymT}_+^{n+1}(\overline{E})$,
we note that the set $\SymMat^{n+1}(\overline{E}, \overline{\X}, ?) \cap 
\SymMat^{n+1}(\overline{E}, 0)$ consists of only one matrix,
which will be denoted as $\overline{\X}^{\cap 0}$.
Similarly to Lemma~\ref{lem:completion}, we can show 
the maximum-determinant property of $\overline{\X}^{\cap 0}$.

\begin{theorem}\label{theo:SOCPcompletion}
The matrix $\overline{\X}^{\cap 0}$ 
has the maximum determinant among all possible matrix completion 
of $\overline{\X} \in \overline{\SymT}_+^{n+1}(\overline{E})$,
i.e.,
\begin{equation}
{\det}_{\SymT} \overline{\X}^{\cap 0}
= \max \left\{ {\det}_{\SymT} \widehat{\X} \ \middle| \ 
\widehat{\X} \in \widehat{\SymT}_+^{n+1}(\overline{E}, \overline{\X}) \right\}.
\end{equation}
\end{theorem}
\begin{proof}
    For $\widehat{\X} \in \widehat{\SymT}_+^{n+1}(\overline{E}, \overline{\X})$, we have
    \begin{align}
    \begin{array}{lcl}
    {\det}_{\SymT} \widehat{\X} &=& \Pi_{(i,j) \in J} 
    (\widehat{X}_{ii} \widehat{X}_{jj} - \widehat{X}_{ij}^2) \\
    &=& \Pi_{(i,j) \in \overline{E}} 
(\widehat{X}_{ii} \widehat{X}_{jj} - \widehat{X}_{ij}^2)
\Pi_{(i,j) \notin \overline{E}} 
(\widehat{X}_{ii} \widehat{X}_{jj} - \widehat{X}_{ij}^2)  \\
    &=& \Pi_{(i,j) \in \overline{E}} 
(\overline{X}_{ii} \overline{X}_{jj} - \overline{X}_{ij}^2)
\Pi_{(i,j) \notin \overline{E}} 
(\overline{X}_{ii} \overline{X}_{jj} - \widehat{X}_{ij}^2)  \\
    &\le& \Pi_{(i,j) \in \overline{E}} 
(\overline{X}_{ii} \overline{X}_{jj} - \overline{X}_{ij}^2)
\Pi_{(i,j) \notin \overline{E}} 
(\overline{X}_{ii} \overline{X}_{jj})  \\
    &=& \Pi_{(i,j) \in \overline{E}} 
(\overline{X}_{ii}^{\cap 0} \overline{X}^{\cap 0}_{jj} - (\overline{X}^{\cap 0}_{ij})^2)
\Pi_{(i,j) \notin \overline{E}} 
(\overline{X}^{\cap 0}_{ii} \overline{X}^{\cap 0}_{jj} - (\overline{X}^{\cap 0}_{ij})^2)  \\
&=& {\det}_{\SymT} \overline{\X}^{\cap 0}.
 \end{array}
\end{align}
For the third equality, we have used $\widehat{X}_{ij} = \overline{X}_{ij}$ 
for $(i,j) \in \overline{E} \cup D$. The fourth equality is derived from
$\overline{\X}^{\cap 0} \in \SymMat^{n+1}(\overline{E}, 0)$.    
Note that 
$\overline{X}_{ii} \overline{X}_{jj} - \widehat{X}_{ij}^2 \ge 0$ holds 
for $(i,j) \notin \overline{E}$ since $\widehat{X}_{ij} \in R_{ij}(\overline{\X})$.
\end{proof}

Theorems~\ref{theo:SOCPsparsity} and \ref{theo:SOCPcompletion} show that 
an optimal solution of \eqref{SOCPrelax} as $\overline{\X}^{\cap 0}$ can be obtained 
by solving \eqref{SOCPsparsity} and substituting 0 in the elements in $J \backslash \overline{E}$.
In view of computational time, 
we can expect that \eqref{SOCPsparsity} is more efficient 
than \eqref{SOCPrelax}, since the number of second-order constraints 
in \eqref{SOCPsparsity} is less than that of 
\eqref{SOCPrelax}. 
In Section~\ref{sec:experiments}, numerical results will be presented to verify the expected efficiency. 

Compared with the matrix completion in SDP, the matrix completion in SOCP has two advantages:
First, the constraints in \eqref{SOCPsparsity} is determined by the aggregate sparsity pattern $\overline{E}$, 
therefore, 
the chordal extension $\widehat{E}$ is not necessary.
Since $\widehat{E}$ has more edges than $\overline{E}$ if
$\overline{E}$ is not chordal, it is a clear advantage for \eqref{SOCPsparsity}.
Secondly, the completion procedure in SOCP is to substitute 0 in the elements in $J \backslash \overline{E}$. 
This procedure is much simpler than the procedure of SDP 
where we need to compute matrices recursively with the maximal cliques $\{C_1, \ldots, C_p\}$.

\subsection{Dual problems}\label{sec:dual}

We investigate the relation between \eqref{SOCPrelax} and \eqref{SOCPsparsity}
with 
 their dual problems.
The dual of 
 \eqref{SOCPrelax}  is
\begin{align}
\text{maximize:} \ \ & \xi \nonumber \\
\text{subject to:} \ \ & \Q_0 + \sum_{k=1}^m \Q_k y_k 
- \H_0 \xi - \sum_{(i,j) \in J} \W^{ij} = \O \nonumber \\
& \y \in \Real_+^m, \ \xi \in \Real \nonumber \\ 
& \W^{ij} \in \SymMat_{+}^{n+1, \{i,j\}}  \ \text{for} \  (i,j) \in J, 
\label{dualSOCPrelax}
\end{align}
where
\begin{equation}
\SymMat_{+}^{n+1, \{i,j\}} := 
\left\{\W \in \SymMat^{n+1}_+ \| 
W_{kl} = 0 \ \text{for} \ (k,l) \notin \{(i,i), (i,j), (j,i), (j,j)\}
\right\}.
\end{equation}
Similarly, the dual of \eqref{SOCPsparsity} is:
\begin{align}
\text{maximize:} \ \ & \xi \nonumber \\
\text{subject to:} \ \ & \Q_0 + \sum_{k=1}^m \Q_k y_k 
- \H_0 \xi - \sum_{(i,j) \in \overline{E}} \W^{ij} - 
\sum_{i\in V_e} w_i \e_i \e_i^T = \O \nonumber \\
& \y \in \Real_+^m, \ \xi \in \Real \nonumber \\ 
& \W^{ij} \in \SymMat_{+}^{n+1, \{i,j\}} \ \text{for} \  (i,j) \in \overline{E} \\
& w_i \ge 0 \ \text{for} \ i \in V_e,
\label{dualSOCPsparisity}
\end{align}
where $\e_i$ is the $i$th unit vector in $\Real^{n+1}$.

In Theorem~4, we have discussed the relation between the primal problems
\eqref{SOCPrelax} and \eqref{SOCPsparsity}.
We show the relation between the two dual problems  in the following 
theorem.

\begin{theorem}\label{theo:dualSOCPsparsity}
Each 
feasible solution of  \eqref{dualSOCPrelax} (or \eqref{dualSOCPsparisity}) can be converted to 
a feasible solution of \eqref{dualSOCPsparisity} (or \eqref{dualSOCPrelax}, respectively) 
while maintaining the objective value. 
\end{theorem}

\begin{proof}
Suppose that $\widehat{\y}, \widehat{\xi}$ and $\widehat{\W}^{ij}$ for $(i,j) \in J$ is a feasible solution of \eqref{dualSOCPrelax}.
Let $\widehat{\Q}:= \Q_0 + \sum_{k=1}^m \Q_k \widehat{y}_k
- \H_0 \widehat{\xi}$,
then it holds that $\widehat{\Q} \in \SymMat^{n+1}(\overline{E}, 0)$.
This indicates that if $(i,j) \in J \backslash \overline{E}$, then $\widehat{\W}^{ij}$ is a diagonal matrix.
Therefore, we can find appropriate $\overline{\W}^{ij} \in \SymMat_+^{n+1,\{i,j\}}$ for each $(i,j) \in \overline{E}$ such that 
$\widehat{\Q} - \sum_{(i,j) \in \overline{E}} \overline{\W}^{ij}$ 
is a diagonal matrix and nonnegative diagonal elements appear only in $V_e$,
by distributing the nonnegative diagonal elements of
$\widehat{\W}^{ij}$ of $(i,j) \in J \backslash \overline{E}$ to 
some $\widehat{\W}^{ij}$ of $(i,j) \in \overline{E}$.
If we denote the nonnegative diagonal elements by $\widehat{w}_i$ for $i \in V_e$,
we know that 
the solution with 
$\widehat{\y}, \widehat{\xi}$, $\overline{\W}^{ij}$ for $(i,j) \in \overline{E}$
and $\widehat{w}_i$ for $i \in V_e$ is a feasible solution
of \eqref{dualSOCPsparisity}. 

The opposite direction can be derived from the fact that 
\begin{equation}
\left\{\sum_{(i,j) \in \overline{E}} \W^{ij} 
+ \sum_{i\in V_e} w_i \e_i \e_i^T  \middle| 
\begin{array}{l}
\W^{ij} \in \SymMat_{+}^{n+1} \cap \SymMat^{n+1, \{i,j\}} \ \text{for} \  (i,j) \in \overline{E}, \\
w_i \ge 0 \ \text{for} \ i \in V_e 
\end{array}
\right\}
\end{equation}
is a subset of 
\begin{equation}
\left\{\sum_{(i,j) \in J} \W^{ij} 
+ \sum_{i\in V_e} w_i \e_i \e_i^T  \ \middle| \ 
\W^{ij} \in \SymMat_{+}^{n+1, \{i,j\}} \ \text{for} \  (i,j) \in J\right\}.
\end{equation}
In fact, if $\overline{\W}^{ij}$ for $(i,j) \in \overline{E}$ and 
$\overline{w}_i$ for $i \in V_e$ are a feasible  solution of \eqref{dualSOCPsparisity},
a feasible solution of \eqref{dualSOCPrelax} can be constructed by assigning the values as 
\begin{align}
\W^{ij} = \left\{
\begin{array}{lcl}
 \overline{\W}^{ij} + \overline{w}_i \e_i\e_i^T &\text{if} & (i,j) \in \overline{E} \ \text{and} \ i \in V_e \\
\overline{\W}^{ij} &\text{if} & (i,j) \in \overline{E} \  \text{and} \ i \notin V_e  \\
 \overline{w}_i \e_i\e_i^T &\text{if} & (i,j) \notin \overline{E} \ \text{and} \ i \in V_e \\
\O &\text{if} & (i,j) \notin \overline{E} \ \text{and} \ i \notin V_e,
\end{array}
\right.\label{eq:SOCPrecover}
\end{align}
and keeping the values of the other variables.
\end{proof}
A direct consequence of Theorem \ref{theo:dualSOCPsparsity} is that  
we can solve \eqref{dualSOCPsparisity} instead of solving \eqref{dualSOCPrelax} and recover the optimal solution
of \eqref{dualSOCPrelax} using \eqref{eq:SOCPrecover}.


\section{Numerical experiments}
\label{sec:experiments}

Numerical experiments on exploiting the sparsity of the SOCP relaxation were conducted with
two test problems, the lattice and  pooling problem. For the test problems,   the SOCP relaxation~\eqref{SOCPrelax} 
is known to  provide
the same optimal value as that of the SDP relaxation~\eqref{SDPrelax} due to the structure of the data matrices.
The aim of the numerical experiments is to observe
how much computational time can be reduced by exploiting the aggregate sparsity of the SOCP relaxation
for the same optimal values by the SDP and SOCP relaxation.


The numerical experiments were performed using MATLAB R2018a
on a MacBook Pro with an Intel Core i7 processor (2.8GHz) and 16 GB memory space. 
The  relaxation problems of the lattice and pooling problem were solved by SeDuMi~\cite{sturm1999using} and 
 MOSEK version 8.1.0.72~\cite{mosek2015mosek}, respectively.
 SeDuMi could not handle the pooling problem as it required too much memory. 
To test with SeDuMi and MOSEK, the  SDP and SOCP relaxations of the test problems should be converted into the following input format:
\begin{align} \nonumber
\begin{array}{rllrll}
(P) & \text{minimize:} \ \ & \c^\top \x & (D) & \text{maximize:} \ \ & \b^\top \y \\
& \mbox{subject to:} \ \ &\mathcal{A}\x = \b, \x \in \mathcal{K}  &	&    \mbox{subject to:} \ \ &\c - \mathcal{A}^\top \y \in \mathcal{K}.
\end{array}
\end{align} 
Here, 
$\mathcal{K}$ stands for a Cartesian product of a linear cones, second-order cones and 
positive semidefinite cones. $\mathcal{\A}$ is a linear map and its 
adjoint operator is denoted with $\mathcal{\A}^\top$.
Note that  the SDP relaxation problems  (\eqref{SDPrelax} and  \eqref{SDPrelaxDC}) and
the SOCP relaxation problems (\eqref{SOCPrelax} and \eqref{SOCPsparsity}) can be formulated as either $(P)$ or $(D)$.

%



To denote which of $(P)$ or $(D)$ is used, we use the notation described in Table~\ref{table:shortname}
for the relaxation problems in the primal form $(P)$.
\begin{table}[tbp]
	\caption{Relaxation problems in the primal form $(P)$} \label{table:shortname}
	\begin{tabular}{l|l}
		\hline
		F-SOCP ($P$) & the full SOCP relaxation \eqref{SOCPrelax} formulated as ($P$) \\
		S-SOCP ($P$) & the sparse SOCP relaxation \eqref{SOCPrelax} formulated as ($P$) \\
		F-SDP ($P$) & the full SDP relaxation \eqref{SDPrelax} formulated as ($P$) \\
		S-SDP ($P$) & the sparse SDP relaxation formulated as ($P$) \\
		\hline 
	\end{tabular}
\end{table}
Similarly, for the dual form $(D)$,  F-SOCP ($D$), S-SOCP ($D$),
F-SDP ($D$), and S-SDP ($D$) are used.
We mention that  S-SDP ($P$) and S-SDP ($D$), which employ the chordal sparsity discussed in Section~\ref{sec:pre},
can be obtained   by
executing SparseCoLO~\cite{fujisawa2009user}.

\subsection{The lattice problem}
We first describe how to generate
a QCQP for the lattice problem with the size $n=n_L^2$ where  $n_L$ is a positive integer. 
Consider
\begin{align}
\mbox{minimize:} \ \ &\x^\top \P_{0} \x  \nonumber \\
\mbox{subject to:} \  \ &\x^\top \P_{k} \x  + r_k \leq 0  \ ( k=1, ... , m). \label{latticeQCQP}
\end{align}
Here $\x \in \Real^n$ is the decision variable.
For the coefficient matrices $\P_0, \P_1, \ldots, $ $\P_m \in \SymMat^{n}$,  a lattice graph 
$G(V_L, E_L)$  is considered with the vertex set $V_L = \{1, 2, \ldots, n_L^2\}$ 
and the edge set 
\begin{equation}
\begin{array}{rcl}
E_L &=& \left\{ ( (i_L-1)n_L +j_L , (i_L-1)n_L +(j_L+1) ) \middle| i_L =1,2,...n_L , j_L =1,2,...,n_L-1 \right\} \\
& & \cup \left\{ ( (i_L-1)n_L +j_L, i_L n_L +j_L) \middle| i_L =1,2,...,n_L-1 , j_L=1,2,...,n_L \right\}.
\end{array}
\end{equation}
Figure~\ref{fig:lattice} illustrates the lattice graph $G(V_L, E_L)$ with $n_L = 4$.
Test instances with the  lattice graph were also used in \cite{yamashita2015fast}.

For $k=0,2,\ldots,m$, 
$[\P_k]_{ij}$ was generated randomly in the interval $[-1, 0]$
for each $(i,j) \in E_L$, and the diagonal elements
$[\P_k]_{ii}$ was generated randomly in the interval $[-1, 1]$ for each $i = 1, \ldots, n+1$.
For $k=1$,  randomly generated number in the interval $(0, 1]$ 
was used as $[\P_1]_{ii}$ for $i = 1, \ldots, n+1$.
The other elements in $\P_0, \ldots, \P_m$ were set to zeros.
For $r_1, \ldots, r_k$, 
negative random  values were used
so that $\x = \0$ could be feasible for the lattice QCQP \eqref{latticeQCQP}.
The feasible region of \eqref{latticeQCQP} is bounded 
since $\P_1$ is a positive diagonal matrix.
Note that 
the assumption in
Theorems~3.4 and 3.5 of \cite{kim2003exact} is satisfied by the non-positivity of the off-diagonal elements, thus
we can obtain the exact optimal value of the lattice QCQP \eqref{latticeQCQP}
by the SDP relaxation \eqref{SDPrelax} or the SOCP relaxation \eqref{SOCPrelax}.

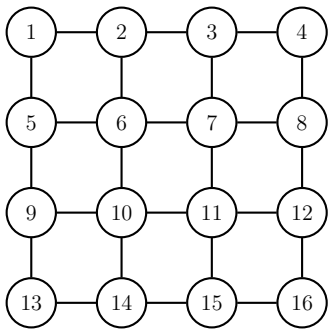
\begin{figure}[tbp]
	 \centering
	 \begin{tikzpicture}
	 [roundnode/.style ={circle,draw,minimum size=3mm, inner sep=2pt, text width=7mm, align=center},thick,scale=0.6, every node/.style={scale=0.7}]
	 \node[roundnode] (1) at(0,0){1};
	 \node[roundnode] (2) at(2,0){2};
	 \node[roundnode] (3) at(4,0){3};
	  \node[roundnode] (4) at(6,0){4};
	  \node[roundnode] (5) at(0,-2){5};
	 \node[roundnode] (6) at(2,-2){6};
	 \node[roundnode] (7) at(4,-2){7};
	  \node[roundnode] (8) at(6,-2){8};
	  \node[roundnode] (9) at(0,-4){9};
	 \node[roundnode] (10) at(2,-4){10};
	 \node[roundnode] (11) at(4,-4){11};
	  \node[roundnode] (12) at(6,-4){12};
	   \node[roundnode] (13) at(0,-6){13};
	 \node[roundnode] (14) at(2,-6){14};
	 \node[roundnode] (15) at(4,-6){15};
	  \node[roundnode] (16) at(6,-6){16};
	  
	  \draw (1) to (2) to (3) to (4)
	  	   (5) to (6) to (7) to (8)
		   (9) to (10) to (11) to (12)
		   (13) to (14) to (15) to (16)
		   (1) to (5) to (9) to (13)
		   (2) to (6) to (10) to (14)
		   (3) to (7) to (11) to (15)
		   (4) to (8) to (12) to (16);
	\end{tikzpicture}
	\caption{A lattice graph for $n_L = 4$.}\label{fig:lattice}
\end{figure}

In Table~\ref{table:1}, we compare  F-SOCP $(P)$ and S-SOCP $(P)$ to see 
the effect of exploiting sparsity in the SOCP relaxation formulated as $(P)$.
The first and second columns show
 the number of constraints and variables, respectively. 
 The third and fourth columns report the computational time in seconds 
and the objective values obtained.
    
	\begin{table}[tbp]
	\caption{Comparing F-SOCP ($P$) and S-SOCP ($P$)}  \label{table:1}
	\begin{adjustwidth}{-.5in}{-.5in}
	\centering
	\scriptsize
	\resizebox{\textwidth}{!}{
	\begin{tabular}{|c|c|c|c|c|c|}
	\hline
	
	Number of &Number of &\multicolumn{2}{c|}{CPU time (s)} &\multicolumn{2}{c|}{Objective value} \\
	\cline{3-6}
	 constraints ($m$) &variables ($n = n_L^2$) & F-SOCP ($P$) & S-SOCP ($P$) & F-SOCP ($P$) & S-SOCP ($P$) \\
%
	\hline
	15 & $64^2$ & 13.78 & 0.20 & -556.56 & -556.56 \\
	\hline
	20 & $81^2$ & 88.91 & 0.12 & -415.75 & -415.75 \\
	\hline
	30 & $100^2$ & 315.61 & 0.20 & -996.37 & -996.37 \\
	\hline
	40 & $121^2$ & 1292.46 & 0.33 & -2481.84 & -2481.84 \\
	\hline
	50 & $144^2$ & 4594.44 & 0.50 & -1157.12 & -1157.12 \\
	\hline
	120 & $169^2$ & 15397.13 & 1.71 & -3308.35 & -3308.35 \\
	\hline
	\end{tabular}
	}
	\end{adjustwidth}	
\end{table}
	
We also observe in Table~\ref{table:1} that 
F-SOCP $(P)$ and S-SOCP $(P)$ attained the same objective value 
as shown in Theorem~\ref{theo:SOCPsparsity}, but
S-SOCP $(P)$ is much more efficient for solving the SOCP relaxation problems than F-SOCP $(P)$.  
The number of  second-order cones for the constraint $\X \in \SymT_+^{n+1}$ in \eqref{SOCPrelax} is
 $\frac{n(n-1)}{2} = \frac{n_L^2 (n_L^2 -1)}{2}$, which is the fourth order of $n_L$.
In contrast, the number of second-order cones for  $\X \in \overline{\SymT}_+^{n+1}(\overline{E})$ in \eqref{SOCPsparsity}
 is $2n_L (n_L-1)$, which is the quadratic order of $n_L$. We see that
the difference in the numbers of second-order cones resulted in the large difference in CPU time shown in Table~\ref{table:1}.
	
\begin{table}[tbp]
\caption{Numerical results on F-SDP $(P)$, S-SDP $(P)$, and S-SOCP $(P)$}
\label{table:2}
	\begin{adjustwidth}{-.5in}{-.5in}
	\centering
	\resizebox{\textwidth}{!}{
	\begin{tabular}{|c|c|c|c|c|c|c|c|}
    \hline	
	Number of &Number of &\multicolumn{3}{c|}{CPU time (s)} &\multicolumn{3}{c|}{Objective value} \\
	\cline{3-8}
	 constraints ($m$)&variables ($n = n_L^2$) & F-SDP $(P)$ & S-SDP $(P)$ & S-SOCP $(P)$ & F-SDP $(P)$ & S-SDP $(P)$ & S-SOCP $(P)$ \\
	\hline
	1000 & $81^2$ & 4.79 & 6.77 & 6.85 & -396.15 & -396.15 &-396.15 \\
	\hline
	1000 & $100^2$ & 6.22 & 8.94 & 9.08 & -669.20 & -669.20 &-669.20 \\
	\hline
	1500 & $121^2$ & 19.12 & 26.79 & 27.89 & -2412.98  & -2412.98 &-2412.98\\
	\hline
	1500 & $144^2$ & 23.05 & 32.77 & 35.47 & -1108.57 & -1108.57 &-1108.57 \\
	\hline
	2000 & $169^2$ & 51.86 & 65.82 & 83.68 & -3043.22 & -3043.22 &-3043.22\\
	\hline
	2000 & $196^2$ & 53.92 & 75.04 & 104.88 & -5401.28 & -5401.28 &-5401.28\\
	\hline
	2500 & $225^2$ & 106.07 & 148.99 & 210.25 & -9921.85 & -9921.85 &-9921.85\\
	\hline
	\end{tabular}
	
	}
	\end{adjustwidth}
	\end{table}
	
Table~\ref{table:2} displays the numerical results on F-SDP $(P)$, S-SDP $(P)$
and S-SOCP $(P)$ for larger $n_L$  than Table~\ref{table:1},
and it shows that F-SDP $(P)$ is the most efficient. 
When we formulate the problem with $m=2500$ into $(P)$, the sizes of $\mathcal{A}$
in F-SDP $(P)$, S-SDP $(P)$, and S-SOCP $(P)$ are 
$2500 \times 53125$, $3026 \times 13322$ and $3340 \times 3985$, respectively.
In S-SDP $(P)$, the average size of maximal cliques is 23.42 
and the number of equality constraints 
to be added for the overlaps between cliques 
(the cardinality of $U$ in \eqref{SDPrelaxDC}) is 526.
The  number of the added equality constraints becomes  computational burden when solving S-SDP $(P)$, 
 though we can reduce the number of variables compared with F-SDP $(P)$.
Next, S-SOCP $(P)$ includes many second-order cones. 
The dimension of each second-order cone in S-SOCP $(P)$ is three. We mention that handling many  cones of small size 
is known to be inefficient.
Therefore, in the standard form of $(P)$, S-SOCP $(P)$ is less efficient than F-SDP $(P)$.

Now, we consider the dual $(D)$.
In Table~\ref{table:3}, we compare four relaxations, 
F-SDP $(D)$, S-SDP $(D)$, 
F-SOCP ($D$) and S-SOCP ($D$).
As seen in Tables~\ref{table:1} and \ref{table:2},
the objective values from the  four problems are the same, so 
the objective values are not shown in Table~\ref{table:3}.
The computational advantage by exploiting the aggregate sparsity  in S-SOCP
over F-SOCP 
is not as large as  Table~\ref{table:1}.    
However,  we can still observe the improved efficiency by exploiting the aggregate sparsity. 

When  the problem with $m=60000$ is formulated as the dual $(D)$, 
the sizes of $\mathcal{A}$
in F-SDP $(D)$, S-SDP $(D)$, F-SOCP $(D)$ and S-SOCP $(D)$ are 
$14365 \times 88651$, $3559 \times 67204 $, $14365 \times 102757$ and $481 \times 61105$, respectively.
It is faster to solve S-SDP $(D)$ and S-SOCP $(D)$, which include
 relatively small-sized $\mathcal{A}$, 
than F-SDP $(D)$ and F-SOCP $(D)$.  
We also see that F-SOCP $(D)$  consumes less computational time 
than F-SDP $(D)$, despite larger-sized $\mathcal{A}$ in  F-SOCP $(D)$  than that of F-SDP $(D)$.

\begin{table}[tbp]
	\caption{Computational time (in seconds) for the dual $(D)$}
\label{table:3}
	\begin{adjustwidth}{-.5in}{-.5in}
	\centering
	\footnotesize
	\resizebox{\textwidth}{!}{
	\begin{tabular}{|c|c|c|c|c|c|}
	\hline	
	Number of &Number of &\multicolumn{4}{c|}{CPU time (s)} \\
	\cline{3-6}
	 constraints $(m)$ &variables $(n = n_L^2)$ & F-SDP $(D)$ & S-SDP $(D)$ & F-SOCP $(D)$ &  S-SOCP $(D)$ \\
	\hline
	9000 & $64^2$ & 39.01 & 12.07 & 32.37 & 2.30 \\
	\hline
	10000 & $81^2$ & 179.98 & 21.77 & 80.60 & 3.70 \\ 
	\hline
	15000 & $100^2$ & 724.58 & 38.03 & 185.11 & 7.32  \\
	\hline
	20000 & $121^2$ & 2373.0 & 83.77 & 484.28 & 14.89  \\
	\hline
	30000 & $144^2$ & 6515.9 & 244.11 & 1323.8 & 31.99 \\
	\hline
	60000 & $169^2$ & 19868.23 & 537.48 & 3436.7 & 76.50  \\
	\hline
	\end{tabular}
	}
	\end{adjustwidth}
\end{table}
	
\begin{table}[tbp]
\caption{Numerical comparison between F-SDP $(P)$ and S-SOCP $(D)$}
\label{table:5}
	\begin{adjustwidth}{-.5in}{-.5in}
	\centering
	\scriptsize
	\resizebox{\textwidth}{!}{
	\begin{tabular}{|c|c|c|c|c|c|}
	\hline	
	Number of &Number of &\multicolumn{2}{c|}{CPU time (s)} &\multicolumn{2}{c|}{Objective value} \\
	\cline{3-6}
	 constraints ($m$) &variables ($n = n_L^2$) & F-SDP $(P)$ & S-SOCP $(D)$ & F-SDP $(P)$ & S-SOCP $(D)$ \\
	\hline
	3000 & $81^2$ & 120.60 & 1.32 & -363.22 & -363.22 \\
	\hline
	5000 & $100^2$ & 591.40 & 2.63 & -616.87 & -616.87 \\
	\hline
	5000 & $121^2$ & 663.15 & 3.79 & -2318.78 & -2318.78 \\
	\hline
	10000 & $144^2$ & 5029.6 & 9.15 & -1091.03 & -1091.03 \\
	\hline
	10000 & $169^2$ & 5412.5 & 11.81 & -2901.12 & -2901.12 \\
	\hline
	15000 & $196^2$ & 19033 & 23.45 & -4204.64 & -4204.64 \\
	\hline
	\end{tabular}
	}
\end{adjustwidth}
\end{table}

Finally, we discuss which  relaxation method for the lattice QCQP is the most efficient.
From Tables~\ref{table:2} and \ref{table:3},  we observe that
F-SDP $(P)$ and S-SOCP $(D)$ outperform other methods in $(P)$ and $(D)$, respectively. 
Table~\ref{table:5} summarizes the computational time to  compare F-SDP $(P)$  and S-SOCP $(D)$.
In all  cases, S-SOCP $(D)$ is remarkably faster than F-SDP $(P)$.
%

We have not included the computational time for the completion procedure for
S-SDP  in all tables. As discussed in Section~\ref{sec:mc-socp}, the SOCP relaxation \eqref{SOCPsparsity} does not require a
completion procedure, 
but it  can generate the optimal solution of the lattice QCQP \eqref{latticeQCQP}.
	
\subsection{The pooling problem}
	
In this subsection, we present numerical results on the relaxation problems from the pooling problem studied in
Kimizuka et al.~\cite{kimizuka2018solving}. The pooling problem with time discretization is a mixed-integer nonconvex QCQP, and it is shown as NP-hard \cite{alfaki2012models}. Kimizuka et al. \cite{kimizuka2018solving} solved
the SDP, SOCP and LP relaxation problems for the QCQP
  and applied a rescheduling method to the solution obtained from the relaxation problem to generate a good approximate
   solution for the pooling problem. 
They used the fact that all the diagonal elements of the data matrices $\P_0, \ldots, \P_m$ are always zeros 
in  the QCQP \eqref{QCQP1} obtained from the pooling problem, and proved that 
the SDP, SOCP and LP relaxations of the QCQP provided the same optimal value.
In \cite{kimizuka2018solving}, they generated the SOCP and LP relaxation problems 
using SPOTless \cite{SPOTless2014}. 

	
For our numerical experiments, we used the 10 test instances
from \cite{kimizuka2018solving}, and 
the description of each test instance  can be found in \cite{kimizuka2018solving}.
In Table~\ref{table:7}, we compare CPU time 
consumed by F-SDP $(D)$, S-SDP $(D)$, F-SOCP $(D)$, and S-SOCP $(D)$.
F-SDP $(D)$ and F-SOCP $(D)$ correspond to the SDP and SOCP relaxation problems used in \cite{kimizuka2018solving}.
S-SDP $(D)$ was obtained using SparseCoLO, while 
S-SOCP $(D)$ was formulated using the aggregate sparsity based on \eqref{SOCPsparsity}. 
The primal standard form $(P)$ of the pooling problem is not compared, as SPOTless~\cite{SPOTless2014} always
generate full SDP  or SOCP relaxations in $(P)$,
and extracting the original  aggregate sparsity from the resulting SDP or SOCP
problems was difficult.


In Table~\ref{table:7}, we first observe from the results on Instances 1 and 3 that
it was much more expensive to solve F-SDP $(D)$  than the other three relaxations.
Furthermore, the other instances could not be solved due to out of memory.
By exploiting the chordal sparsity, S-SDP $(D)$ successfully solved most instances except for  the largest two instances
where out of memory occurred.
Compared to the SDP relaxations, 
the SOCP relaxations can be solved very efficiently. 
F-SOCP $(D)$, which does not exploit any sparsity,
is competitive with 
S-SDP $(D)$ in terms of computational time.
In addition, F-SOCP $(D)$ did not take much memory to solve Instances 9 and 10.

For the performance of S-SOCP $(D)$, 
we see that S-SOCP $(D)$ is 3-6 times faster than  F-SOCP $(D)$. In particular, the number of second-order cones  for Instance 10 in F-SOCP $(D)$ was 1306536, while S-SOCP $(D)$ had only 70464 second-order cones.
This difference resulted in much shorter CPU time for S-SOCP (D).
	
\begin{table}[tbp]
	\caption{Numerical results on the pooling problem (OOM stands for out of memory.)}
\label{table:7}
	\begin{adjustwidth}{-.5in}{-.5in}
	\centering
	\resizebox{\textwidth}{!}{
	\Large
	\begin{tabular}{|c|c|c|c|c|c|c|c|c|c|}
	\hline
&Number of  &\multicolumn{4}{c|}{CPU time (s)}  \\
	\cline{3-6}
	Instance &variables ($n$) & F-SDP $(D)$ & S-SDP $(D)$ & F-SOCP $(D)$ & S-SOCP $(D)$  \\
	\hline	
	1 & $138$ & 768.24 & 0.82 & 0.68 & 0.51 \\ \hline
	2 & $298$ & OOM & 1.00 & 1.33& 0.62 \\ \hline
	3 & $176$ & 5024.0 & 1.15 & 0.85 & 0.68 \\ \hline
	4 & $356$ & OOM & 1.67 & 1.96 & 0.60 \\ \hline
	5 & $222$ & OOM & 2.22 & 1.01 & 0.57 \\ \hline
	6 & $446$ & OOM & 4.01 & 2.80 & 0.60 \\ \hline
	7 & $1228$ & OOM & 19.98 & 21.66 & 0.86 \\ \hline
	8 & $1228$ & OOM & 21.83 & 21.02 & 1.04 \\ \hline
	9 & $1334$ & OOM & OOM & 28.44 & 4.84 \\ \hline
	10 & $1616$ & OOM & OOM & 40.90 & 6.53 \\ \hline 
	\end{tabular}
	}
	\end{adjustwidth}
\end{table}

\section{Conclusion}
\label{sec:conclusion}
We have presented a method to exploit
the aggregate sparsity in the SOCP relaxation. From the numerical experiments in Section 4, we
have observed that the proposed
approach is very efficient in solving the SOCP relaxation. 
In addition, the proposed method can obtain the optimal solution with the simple matrix completion that 
attains the maximum determinant in the SOCP relaxation, as in the matrix completion for the SDP relaxation. 

For future work, exploiting sparsity in the SDSOS relaxation~\cite{ahmadi2014dsos} can be considered
to obtain an approximate solution of polynomial optimization problems fast. 
Since the rows and columns of matrices involved in the relaxation problems for polynomial optimization problems usually have some structure,  the proposed approach in this paper can be applied to improve the computational efficiency. 
In addition, the hierarchy that employs
 SOCP for polynomial optimization problems in \cite{kuang2019alternative} will be also our interest.


\ifx01 
\bibliographystyle{abbrv}
\begin{spacing}{0.3}
{\small
}
\end{spacing}
\else
\begin{spacing}{0.3}
\end{spacing}

\fi

\end{document}